\documentclass[11pt,a4paper]{article}
\usepackage{fullpage}
\usepackage{amssymb,amsthm}
\usepackage{amsmath}
\usepackage{amsfonts}
\usepackage{color}
\usepackage[colorinlistoftodos]{todonotes}

\newcommand{\hh}{\textbf{H}}
\newcommand{\ee}{\textbf{E}}
\renewcommand{\L}{\emph{\textbf{L}}}
\newcommand{\Vg}{\textbf{V}(\Ga)}
\newcommand{\MtE}{\text{\boldmath{$S$}}}
\newcommand{\nn}{\text{\boldmath{$\nu$}}}
\newcommand{\vv}{\textbf{v}}
\newcommand{\uu}{\textbf{u}}
\newcommand{\ww}{\textbf{w}}
\newcommand{\ff}{\textbf{f}}

\newcommand{\ga}{\gamma}
\newcommand{\hc}{\emph{\textbf{H}}_{\rot}}
\newcommand{\hcext}{\emph{\textbf{H}}_{\rot}^{\text{ext}}}
\renewcommand{\H}{\emph{\textbf{H}}}

\newcommand{\hcg}{\emph{\textbf{H}}_{\Rot}}
\newcommand{\hdg}{\emph{\textbf{H}}_{\Div}}
\newcommand{\Grad}{\nabla_{\Ga}}

\newcommand{\<}{\langle}
\renewcommand{\>}{\rangle}
\newcommand{\Om}{\Omega}
\newcommand{\Ga}{\Gamma}
\renewcommand{\div}{\mathrm{div}}
\newcommand{\Div}{\mathrm{div}_\Ga}
\newcommand{\rot}{\mathbf{curl}}
\newcommand{\Rot}{\mathrm{curl}_\Ga}
\newcommand{\Rotv}{\rot_\Ga}
\newcommand{\bR}{\mathbb{R}}
\newcommand{\Z}{\mc Z}
\renewcommand{\b}[1]{\overline{#1}}
\newcommand{\Omext}{\Om_\mathrm{ext}}
\newcommand{\mc}[1]{\mathcal{#1}}
\newcommand{\wt}[1]{\widetilde{#1}}
\newcommand{\vp}{\varphi}
\newcommand{\ld}{\lambda}
\newcommand{\bN}{\mathbb{N}}
\newcommand{\bC}{\mathbb{C}}
\newcommand{\ds}{\,ds}
\newcommand{\dx}{\,dx}
\newcommand{\om}{\omega}

\newtheorem{theorem}{Theorem}[section]

\newtheorem{remark}[theorem]{Remark}
\newtheorem{lemma}[theorem]{Lemma}
\newtheorem{definition_en}[theorem]{Definition}
\newtheorem{hypothese}[theorem]{Hypothesis}

\title{The electromagnetic scattering problem with Generalized Impedance Boundary Conditions }
\author{N. Chaulet\footnote{Department of Mathematics, University College London,  
Gower street, London, WC1E 6BT, UK.}}
\date{}

\begin{document}

\maketitle
%\tableofcontents

\begin{abstract}
In this paper we consider the electromagnetic scattering problem by an obstacle characterised by a Generalized 
Impedance Boundary Condition in the harmonic 
regime. These boundary conditions are well known to provide accurate models for thin layers or imperfectly conducting 
bodies. We give two different formulations 
of the scattering problem and we provide some general assumptions on the boundary condition under which the 
scattering problem has at most one solution. We 
also prove that it is well-posed for three different boundary conditions which involve second order surface differential 
operators under weak sign assumptions on 
the coefficients defining the surface operators.\\[2pt]
\textbf{Keywords:} Maxwell's equations, Generalized Impedance Boundary Conditions, Electromagnetic Scattering, Helmholtz' Decomposition
\end{abstract}

\section{Introduction}
Driven by recent advances in the study of inverse acoustic scattering problems in the presence of so-called generalized 
impedance boundary conditions (see 
\cite{BoChHa11,BoChHa12,CakKre13,ChChHa13}) we study in this paper well-posedness of the forward electromagnetic 
scattering problem in the harmonic regime in the case where the 
scatterer is characterised by a boundary condition of the form
\[
\nn \times \ee + \mc{Z} \hh_T = \ff \quad \text{on } \Ga
\]
where $\Ga$ is the boundary of the scatterer, $\nn$ is the outward unit normal vector to $\Ga$, $\ee$ is the electric field, 
$\hh_T$ stands for the tangential component 
of the magnetic field $\hh$, $\mc{Z}$ is a surface differential operator and $\ff$ is a source term.  This kind of boundary 
conditions, often referred to  as Generalized Impedance Boundary  Condition, are known to provide accurate models for 
all sort of small scale structures. Moreover, in some specific configurations, such as the 
scattering by a perfect 
conductor covered by a thin layer of dielectric or of ferromagnetic material  (see 
\cite{BenLem08,DurHadJol06,HadJol02}), or the scattering by an imperfectly 
conducting body (see \cite{HaJoNg08}),  asymptotic analysis techniques provide an expression for $\Z$ in terms of 
surface differentials operators as well as approximation properties.

In this paper we establish sufficient conditions on the operator $\mc{Z}$ under which the scattering problem is well 
posed. We introduce two different ways of 
writing the problem: the first one (which we will call the volume approach) consists in considering the scattering problem 
as a volume problem and in studying the associated variational formulation.
This path is rather standard and follows the lines of \cite[chapter 10]{Mon03}. We state a general existence and uniqueness result for the scattering problem which uses the volume formulation in Theorem \ref{th:div}. Nevertheless, with these standard 
approach one needs to assume some compatibility between the signs of the surface operator $\Z$ and the sign of the 
volume contribution to the variational formulation to ensure reasonable coercivity properties.  Actually, at least for the 
acoustic scattering problem problem (see \cite{ChChHa13,Ver99}), it seems that such restrictive conditions are not 
needed.  To clarify this point, we consider a different formulation for the scattering problem which consists in 
writing the problem as  a single operator equation posed on the 
boundary of the scatterer. We will call this approach the surface formulation. We indeed show that the scattering problem is equivalent to finding the tangential component of 
the electromagnetic field $\hh$ that solves
\[
(\MtE_\Ga + \Z) \hh_T = \ff\quad \text{on } \Ga
\] 
where $\MtE_\Ga$ is the so-called Magnetic-to-Electric Calder\'on operator (see \cite[chapter 9]{Mon03} for example). In 
the scalar case, it is sufficient to assume that $\mc{Z}$ is a pseudo-differential operator of 
order greater or smaller than $1$ to obtain existence and  uniqueness of the solution to the scattering problem. 

Even though we establish a general existence and uniqueness result for this formulation in Theorem \ref{th:existgen},  the 
situation is more challenging than in the scalar case mainly because the principal part of $\Z$ may have a kernel of infinite dimension. To tackle 
this difficulty, we will introduce a tailored Helmholtz' decomposition on the boundary of the scatterer. This allows for 
example to treat the case of an operator $\Z$ corresponding to the first order impedance boundary condition for thin 
coatings which is given by
\[
\Z = \frac{i\delta}{\om\epsilon} \Rotv\Rot - i\om \mu \delta
\]
where $\Rotv$ and $\Rot$ stand for the surface vectorial and scalar rotational operators, $\epsilon$ and $\mu$ are the 
dielectric constants of the coating 
and $\delta$ is the thickness of the layer. For this operator, the surface approach gives well-posedness regardless the 
sign of $\epsilon$ (which can be negative for metals) whereas the volume approach seems to be limited to positive $
\epsilon$.  

In the next section we introduce notations and recall some important concepts for the study of boundary value problems 
for Maxwell's equations. In the third 
section, we introduce the volume and surface equations and we give general results about existence and uniqueness.  
Finally, the fourth section is dedicated 
to the study of well-posedness for three different surface operators of order $2$, each of them requiring the use of 
different techniques.

\section{Problem setting and variational spaces}
\label{sec:GIBCMaxwell}
Let $\Om$ be a simply connected open bounded domain of $\bR^3$ with $C^{1,1}$ boundary $\Ga$ and  let $\Omext := 
\bR^3\setminus\b{\Om}$ be its 
complementary. We consider the following exterior boundary value problem for the electromagnetic field $(\ee,\hh)$ at 
frequency $\om$:
\begin{equation}
\label{pb:exterior}
\begin{cases}
\rot\ \hh +i\om \ee =0 \quad \text{in} \ \Omext , \\
\rot\ \ee-i\om   \hh=0\quad \text{in} \ \Omext , \\
 \nn \times \ee + \Z \hh_T=\ff\quad \text{on } \Ga
\end{cases}
\end{equation}
where $\nn \in (C^{0,1}(\Ga))^3$ is the outward unit normal to $\Om$, $\hh_T:=(\nn\times\hh)\times\nn$,  $\Z$ is a 
surface differential operator (see Definition 
\ref{def:ZMaxwell}) and $\ff$ is some function defined on $\Ga$. When considering the scattering of an incident wave  $
(\ee^i,\hh^i)$ which is solution to
\[
\rot\ \hh^i +i\om \ee^i =0 \quad \text{and}\quad \rot\ \ee^i-i\om   \hh^i=0\quad \text{in} \ \bR^3,
\]
the right hand side $\ff$ is given by 
\[
\ff := -\left(  \nn \times \ee^i + \Z \hh^i_T \right).
\]
We complement equations \eqref{pb:exterior} with the so-called Silver-M\"uller radiation condition 
\begin{equation}
\label{eq:SM}
    \lim \limits_{R \to \infty} \int_{\partial B_{R}}|\hh \times \hat{x} - (\hat{x}\times\ee )\times \hat{x}|^{2}\ds=0
\end{equation}
where $B_R$ is a ball of radius $R$ and $\hat x:=x/|x|$.

To study equations \eqref{pb:exterior}-\eqref{eq:SM} we  introduce some classical energy spaces and to define specific 
surface differential operators.
We recall hereafter some classical results from \cite[chapter 2]{Ces96} for the convenience of the reader.  Let $\mc{O}$ 
be a generic bounded simply connected open set of $\bR^3$ with 
$C^{1,1}$ boundary $\partial \mc{O}$ and with outer unit normal $\nn$. Let us first introduce the usual energy space $
\hc(\Omext)$ of $(L^2(\mc{O}))^3$ 
distributions with $\rot$ in $(L^2(\mc{O}))^3$ as well as the space of $L^2$ tangential vector fields on $\partial \mc O$:
\[
\L^2_t(\partial \mc O):= \{\vv \in (L^2(\partial \mc O))^3 \ |\ \vv\cdot\nn =0\}.
\]
For $s \in [-1,1]$ we denote $ \H^s_t(\partial \mc O)$ the closure of $\{\vv\in (C^\infty(\partial \mc O))^3 \ |\ \vv\cdot\nn =0 
\}$ in $(H^s(\partial \mc O))^3$.
The tangential trace operators are given for $\vv \in( H^1(\mc O))^3$  by
\[
\ga_t \vp := \nn \times \vp|_{\partial \mc O} \ ,\quad \vv_T=\ga_T(\vv) := (\nn \times \vv|_{\partial \mc O})\times \nn.
\]
These two operators are bounded and linear from $(H^1(\mc O))^3$ into $\L^2_t(\partial \mc O)$. Let us now introduce 
the surface differential operators $
\nabla_{\partial \mc O} : H^1(\partial \mc O) \rightarrow \L^2_t(\partial \mc O)$ and $\rot_{\partial \mc O} : H^1(\partial \mc 
O) \rightarrow \L^2_t(\partial \mc O) $ that 
are given for $u\in H^1(\partial \mc O)$ by
\[
\nabla_{\partial \mc O} u := \ga_T( \nabla \wt u) \quad \text{and} \quad \rot_{\partial \mc O} u := -\nn \times \nabla_{\partial 
\mc O} u
\]
where $\wt u$ is some extension of $u$ to a three dimensional neighbourhood of $\partial \mc O$. We denote their 
adjoints  $-{\rm div}_{\partial \mc{O}} :  
\L^2_t(\partial \mc O)\rightarrow H^{-1}(\partial \mc O)$  and  ${\rm curl}_{\partial \mc{O}} : \L^2_t(\partial \mc O) 
\rightarrow H^{-1}(\partial \mc O)$ that are defined 
for all $u \in H^1(\partial \mc O)$ and $\vv \in \L^2_t(\partial \mc O)$ by 
\begin{align*}
\int_{\partial \mc O} \nabla_{\partial \mc{O}} u\cdot \b\vv \ds = - \<u,\div_{\partial \mc{O}}\vv\>_{H^1(\partial \mc O),H^{-1}
(\partial \mc O)} ,\\
\int_{\partial \mc O} \rot_{\partial \mc{O}} u \cdot \b\vv \ds = \<u,{\rm curl}_{\partial \mc{O}}\vv\>_{H^1(\partial \mc 
O),H^{-1}(\partial \mc O)}.
\end{align*}
The vector operators $\rot_{\partial \mc O}$ and $\nabla_{\partial \mc O}$ can be extended to continuous linear operators 
from $H^s(\partial \mc O)$ into $\H^{s-1}
_t(\partial \mc O)$ while the scalar operators $\div_{\partial \mc O}$ and ${\rm curl}_{\partial \mc O}$ can be extended to 
continuous linear operators from $
\H^s_t(\partial \mc O)$ into $H^{s-1}(\partial \mc O)$ for $s\in [3/2,-1/2]$. 
Moreover,
 \begin{equation}
 \label{eq:rotequality}
 {\rm curl}_{\partial \mc O}\uu := \nn \cdot \rot\ \wt \uu  \quad \text{and} \quad \div_{\partial \mc O}\uu = {\rm curl}_{\partial 
\mc O}(\nn\times \uu) 
 \end{equation}
for all $\uu \in\H^1_t(\partial \mc O)$ where $\wt \uu$ is some extension of $\uu$ to a neighbourhood of $\partial \mc O$.

We conclude this section by introducing the following boundary spaces for $s\in [-1/2,1/2]$:
\begin{align*}
\H^{s}_{\div_{\partial \mc O}}(\partial \mc O) := \{ \vv \in \H^{s}_t(\partial \mc O) \ | \ \div_{\partial \mc O}\vv \in H^{s}
(\partial \mc O)\},\\
\H^{s}_{{\rm curl}_{\partial \mc O}}(\partial \mc O) := \{ \vv \in \H^{s}_t(\partial \mc O)\ |\ {\rm curl}_{\partial \mc O} \vv  \in 
H^{s}(\partial \mc O)\}.
\end{align*}
The specific spaces $\H^{-1/2}_{{\rm curl}_{\partial \mc O}}(\partial \mc O)$ and $\H^{-1/2}_{\div_{\partial \mc O}}(\partial 
\mc O)$ are dual to each other $\L^2_t(\partial \mc O)$ as pivot space and we have the following classical result for the 
tangential trace operators.
\begin{theorem}
The trace operators 
\[
 \ga_t \,:\, \hc(\mc O) \rightarrow \H^{-1/2}_{\div_{\partial \mc O}}(\partial \mc O)\; , \quad \ga_T \,:\, \hc(\mc O) \rightarrow  
\H^{-1/2}_{{\rm curl}_{\partial \mc O}}(\partial 
\mc O)
 \]
are linear continuous and surjective and the following  formula holds for any functions $\uu$ and $\vv$ in $\hc(\mc O)$
 \begin{equation}
 \label{eq:DUAL}
 \int_{\mc O} (\rot\ \uu \cdot \b{ \vv} - \uu\cdot\rot\ \b{\vv})\dx = \<\ga_t(\uu),\ga_T(\vv)\>_{\H^{-1/2}_{\div_{\partial \mc O}}
(\partial \mc O),\H^{-1/2}_{{\rm curl}_{\partial 
\mc O}}(\partial \mc O)}.
 \end{equation}
\end{theorem}

%\begin{remark}
%All these definitions can be extended to the case of Lipschitz domains and the trace theorem above remains valid (see 
%\cite{BufCia01a,BufCia01b,BuCoSh02}).
%\end{remark}

\section{Study of an abstract boundary value problem}
\label{sec:pbdirectMaxwell}
 Let us denote by $\Vg\subset \L_t^2(\Ga)$ endowed with its inner product $(\cdot,\cdot)_{\Vg}$ a Hilbert space that is 
such that
\begin{equation*}
\{\vv\in (C^\infty(\Ga))^3 \ | \ \vv\cdot\nn=0\} \subset \Vg
\end{equation*}
and such that the injection is dense. Let us denote $\Vg^*$ the dual space of $\Vg$ with respect to $\L^2_t(\Ga)$. The 
impedance operator is defined as follow.
\begin{definition_en}
\label{def:ZMaxwell}
A generalised impedance operator $\Z$  is a linear and bounded operator from $\Vg$ into its dual $\Vg^*$.
\end{definition_en}
Let us define
\[
 \hcext(\Omext) := \{ \vv \in (\mc{D}'(\Omext))^3\ |\ \vp \vv \in \hc(\Omext) \text{ for all } \vp \in \mc{D}(\bR^3)\}
\]
and $V_{\hh}:=\{\hh \in \hcext(\Omext) \ | \ \hh_T \in \Vg\}$, the exterior problem \eqref{pb:exterior} together with the 
radiation condition \eqref{eq:SM} then writes for 
$\ff \in \Vg^*$:
\begin{equation}
\label{pb:GIBCrotVol}
\begin{cases}
\text{Find } (\ee,\hh) \in   \hcext(\Omext) \times V_{\hh} \text{ such that} \\
\rot\ \hh+i\om \ee =0 \quad \text{in}\  \Omext, \\
\rot\ \ee-i\om\hh=0 \quad \text{in} \ \Omext, \\
\nn\times\ee+\Z \hh_T = \ff \quad \text{on} \ \Ga,\\
\displaystyle  \lim \limits_{R \to \infty} \int_{\partial B_{R}}|\hh\times \hat{x} - (\hat{x}\times\ee)\times \hat{x}|^{2}\ds=0
\end{cases}
\end{equation}
To study existence and uniqueness of the solution to \eqref{pb:GIBCrotVol} we have to reformulate these equations in a 
bounded domain. In the following we 
propose two different approaches to achieve this goal. The first is rather classical and consists in bounding the domain $
\Omext$ by introducing a ball that contains 
the domain $\Om$ and by applying a transparent boundary condition on this artificial boundary. The second approach 
consist in writing the system 
\eqref{pb:GIBCrotVol} as a single equation on $\Ga$ by using the so-called Magnetic-to-Electric Calder\'on operator for 
the exterior problem. In Lemma \ref{le:EquiVolSurf} we prove that  these 
two formulations are equivalent.

\subsection{A volume formulation in a bounded domain}
\label{sec:volbouded}
Let $B_R$ be a ball of radius $R$ such that $\Om\subset B_R$, and let us introduce the Magnetic-to-Electric Calder\'on 
operator   $\MtE_R : \H^{-1/2}
_{\div_{\partial B_R}}(\partial B_R) \rightarrow \H^{-1/2}_{\div_{\partial B_R}}(\partial B_R)$ defined for  $\vv \in\H^{-1/2}
_{\div_{\partial B_R}}(\partial B_R) $ by $
\MtE_R\vv := \hat{x}\times \ee$ where $(\ee,\hh) \in \hcext(\bR^3\setminus \b{B_R})\times \hcext(\bR^3\setminus 
\b{B_R})$ is the unique solution (see \cite{ColKre13} for fundamental results about electromagnetic scattering theory) to  
 \begin{equation*}
%\label{pb:MaxwellDir}
\begin{cases}
\rot\ \hh+i\om \ee =0 \quad \text{in}\ \bR^3\setminus \b{B_R}, \\
\rot\ \ee-i\om\hh=0 \quad \text{in} \ \bR^3\setminus \b{B_R}, \\
\hat{x}\times\hh = \vv  \quad \text{on}\  \partial B_R,\\
 \displaystyle\lim \limits_{r \to \infty} \int_{\partial B_{r}}|\hh\times \hat{x} - (\hat{x}\times\ee)\times \hat{x}|^{2}\ds=0.
\end{cases}
\end{equation*}
Let us denote $\Om_R:= B_R\setminus\b{\Om}$ and let us define the Hilbert space $V_{\hh,R}:= \{\vv\in \hc(\Om_R)\ |\  
\vv_T \in \Vg\}$ endowed with the norm
\[
\|\cdot\|_{V_{\hh,R}}:= \|\cdot\|_{\hc(\Om_R)} + \|\cdot\|_{\Vg}.
\]  
Then, for any $\ff\in \Vg^*$, problem \eqref{pb:GIBCrotVol} is equivalent to:
\begin{equation}
\label{pb:GIBCBR}
\begin{cases}
\text{Find}\ (\ee,\hh)\in \hc(\Om_R)\times V_{\hh,R}\ \text{such that} \\
\rot\ \hh+i\om \ee =0 \quad \text{in }  \Om_R, \\
\rot\ \ee-i\om\hh=0 \quad \text{in } \Om_R, \\
\nn\times\ee+\Z \hh_T =\ff \quad \text{on} \ \Ga,\\
\hat x\times\ee -\MtE_R(\hat{x}\times\hh)=0 \quad  \text{on } \partial B_R
\end{cases}
\end{equation}
which is equivalent to find $\hh \in V_{\hh,R}$ such that
\begin{equation}
\label{eq:varformglo}
\begin{aligned}
\int_{\Om_R} \rot\ \hh \cdot \rot\ \b\vv  -\om^2\hh\cdot \b\vv\dx -i\om\<\Z \hh, \vv\>_{\Vg^*,\Vg}&-i\om\int_{\partial B_R} 
\MtE_R(\hat{x}\times\hh)\cdot \b\vv\ ds \\
& = - i\om \< \ff,\vv\>_{\Vg^*,\Vg}
\end{aligned}
\end{equation}
for all $\vv \in V_{\hh,R}$. To ensure weak coercivity  of this variational formulation one has to assume that 
the imaginary part of $\Z$ is negative. In 
fact, this is not always  necessary and we overcome this difficulty by introducing an alternative formulation for 
problem \eqref{pb:GIBCrotVol} in next section.

\subsection{A surface formulation}
Let us introduce the so-called Magnetic-to-Electric Calder\'on operator  $\MtE_\Ga : \H^{-1/2}
_{\Rot}(\Ga) \rightarrow \H^{-1/2}_{\Div}(\Ga)$ defined for $\vv \in  \H^{-1/2}_{\Rot}(\Ga)$ by  $\MtE_\Ga \vv:= \nn\times 
\ee$ where $(\ee,\hh) \in \hcext(\Omext)\times 
\hcext(\Omext)$ is the unique solution to
\begin{equation}
\label{pb:MaxwellDir}
\begin{cases}
\rot\ \hh+i\om \ee =0 \quad \text{in }  \Omext, \\
\rot\ \ee-i\om\hh=0 \quad \text{in }  \Omext, \\
\hh_T = \vv  \quad \mbox{on }  \Ga,\\
 \displaystyle\lim \limits_{R \to \infty} \int_{\partial B_{R}}|\hh\times \hat{x} - (\hat{x}\times\ee)\times \hat{x}|^{2}\ds=0
\end{cases}
\end{equation}
and we recall that this operator is linear and continuous (see \cite{Mon03} for more details). Using $\MtE_\Ga$, problem 
\eqref{pb:GIBCrotVol} can be rewritten in 
these terms:
\begin{equation}
\label{pb:Surface}
\begin{cases}
\text{Find } \uu \in \Vg \cap \hcg^{-1/2}(\Ga)  \text{ such that} \\
 \MtE_{\Ga}(\uu) + \Z\uu = \ff
 \end{cases}
\end{equation}
for $\ff\in (\Vg\cap \hcg^{-1/2}(\Ga))^*$ where $ \Vg \cap \hcg^{-1/2}(\Ga)$ is endowed with the norm
\[
\|\cdot\|_{ \Vg \cap \hcg^{-1/2}(\Ga)} := \|\cdot \|_{\Vg} + \|\cdot\|_{\hcg^{-1/2}(\Ga)}.
\]
Equation \eqref{pb:Surface} makes sense in $(\Vg\cap \hcg^{-1/2}(\Ga))^*$ since $\MtE_\Ga(\uu) \in \H^{-1/2}_{\Div}(\Ga)
$ and this space can be identified with the dual space of  $
\H^{-1/2}_{\Rot}(\Ga)$. Therefore, $\MtE_\Ga(\uu)$ can be identified with a linear and continuous application on $
\H^{-1/2}_{\Rot}(\Ga)$ and consequently on 
$(\Vg\cap \hcg^{-1/2}(\Ga))$.
As stated in the next Lemma, problems \eqref{pb:GIBCrotVol} and \eqref{pb:Surface}  are equivalent.
\begin{lemma}
\label{le:EquiVolSurf}
Let  $\ff$ be in $ (\Vg\cap \hcg^{-1/2}(\Ga))^*$, if $\uu \in \Vg\cap \hcg^{-1/2}(\Ga)$ solves (\ref{pb:Surface}) then  the 
unique solution $(\ee,\hh) \in \hc^\text{ext}(\Omext)\times V_{\hh} $ to 
\eqref{pb:MaxwellDir} with $\vv=\uu$ solves (\ref{pb:GIBCrotVol}).  Conversely, if  $(\ee,\hh)  \in \hc^\text{ext}(\Omext)
\times V_{\hh}$ solves (\ref{pb:GIBCrotVol}) 
then $\hh_{T} \in \Vg\cap \hcg^{-1/2}(\Ga)$ and solves (\ref{pb:Surface}).
\end{lemma}
\begin{proof}
We take $f\in (\Vg\cap \hcg^{-1/2}(\Ga))^*$ and let $\uu \in \Vg\cap \hcg^{-1/2}(\Ga)$ be a solution to (\ref{pb:Surface}). 
We define $(\ee,\hh) \in \hc^\text{ext}(\Omext)\times V_{\hh}$ as being the unique 
solution to  \eqref{pb:MaxwellDir} for $\vv = \uu$ on $\Ga$.  The tangential component of $\hh$ satisfies  $\hh_T=\uu$ on 
$\Ga$ and then $\hh \in V_{\hh}$. Finally, 
since $ \nn\times \ee = \MtE_\Ga (\uu) $ and since $\uu$ solves (\ref{pb:Surface}) we obtain 
\[
\nn\times \ee + \Z \hh_T = f   \quad \text{on } \Ga
\]
which means that  $(\ee,\hh)$  solves \eqref{pb:GIBCrotVol}. 

The reverse statement is straightforward since $ \nn \times \ee=\MtE_{\Ga}(\hh_{T}) $  as soon as $(\ee,\hh)$ solves 
Maxwell's equations outside $\Om$.
\end{proof} 
We therefore obtain that problems \eqref{pb:GIBCrotVol}, \eqref{pb:GIBCBR} and \eqref{pb:Surface} are equivalent. We 
establish now a well-posedness result  for \eqref{pb:GIBCrotVol} which is valid for a general class of operators $\Z$.

\subsection{Existence and uniqueness for a general class of boundary conditions}
First, to ensure uniqueness, we impose a certain absorption condition to be satisfied by the boundary operator $\Z$. In general 
this hypothesis is not restrictive since it 
is linked to some absorption property of the modelled material.
\begin{hypothese}
\label{hyp:uniciteMaxwell}
The operator $\Z$ has a non negative real part, that is:
\[
 \Re \< \Z \vv,\vv\>_{\Vg^*,\Vg} \geq 0
\]
for all $\vv \in \Vg$.
\end{hypothese}
Under this hypothesis we prove uniqueness.
\begin{theorem}
\label{th:MaxUnicite}
If Hypothesis \ref{hyp:uniciteMaxwell} is satisfied then problem (\ref{pb:GIBCrotVol}) has at most one solution.
\end{theorem}
\begin{proof}
Assume that $(\ee,\hh)\in \hcext(\Omext)\times V_\hh$ satisfies \eqref{pb:GIBCrotVol} with $\ff=0$ on $\Ga$. Let $B_R$ 
be a ball of radius $R$ that contains $\b 
\Om$, by using the integration by part formula formula \eqref{eq:DUAL} in $\Omext \cap B_R$ we find that $(\ee,\hh)$ 
satisfies
\begin{align*}
\< \ga_t \ee ,\ga_T \hh \>_{\hdg^{-1/2}(\Ga),\hcg^{-1/2}(\Ga)}-\int_{\partial B_R} \hat{x} \times \ee \cdot\b{\hh} \ds&= 
\int_{\Omext\cap B_R} -\rot\ \ee\cdot \b{\hh} + \ee
\cdot\rot\ \b{\hh} \dx \\
&= -\int_{\Omext\cap B_R} i\om |\hh|^2 +i\om |\ee|^2 \dx.
\end{align*}
By taking the real part of this equality we obtain
\[
\Re\< \ga_t \ee ,\ga_T \hh \>_{\hdg^{-1/2}(\Ga),\hcg^{-1/2}(\Ga)} = \Re\left(\int_{\partial B_R} \hat{x} \times \ee \cdot\b{\hh} 
\ds\right)
\]
and since $ \nn\times \ee = - \Z \hh_T$ on $\Ga$ this relation becomes
\[
-\Re\left\< \Z \hh_T,\hh_T \right \>_{\Vg^*,\Vg}= \Re\left(\int_{\partial B_R} \hat{x} \times \ee \cdot\b{\hh} \, ds\right).
\]
Since we assume that the real part of $\Z$ is non negative, this gives
\[
 \Re\left(\int_{\partial B_R} \hat{x} \times \ee \cdot\b{\hh} \, ds\right) \leq 0
\]
which in regards of Rellich's Lemma (\cite[lemme 9.28]{Mon03}) gives $\ee=\hh=0$ in $\bR^3\setminus\b{B_R}$. The 
unique continuation principle then gives $\ee=
\hh=0$ in $\Omext$ and this concludes the proof.
\end{proof}
As a consequence, to prove that problem \eqref{pb:GIBCrotVol} is well-posed it is sufficient to prove that it can be 
formulated as a  Fredholm type problem. When  $\Vg$ is compactly embedded into $
\hcg^{-1/2}(\Ga)$ the surface formulation \eqref{pb:Surface} allows to prove this property in a straightforward way as soon as  $\Z : 
\Vg \rightarrow \Vg^*$ can be decomposed as the sum of an isomorphism and a compact operator.
\begin{theorem}
\label{th:existgen}
Let $\Z$ be an impedance operator such that Hypothesis \ref{hyp:uniciteMaxwell} is satisfied. If $\Vg$ is compactly 
embedded into $\hcg^{-1/2}(\Ga)$ and $\Z = 
\mc{T} + \mc{K}$ where $\mc{T} : \Vg \rightarrow \Vg^*$ is an isomorphism and $\mc{K} : \Vg \rightarrow \Vg^*$ is a 
bounded and compact operator, then for all $\ff 
\in \Vg^*$ problem \eqref{pb:GIBCrotVol} has a unique solution $(\ee,\hh) \in  \hcext(\Omext) \times V_{\hh} $ and for all ball 
$B_R$ that contains $\b \Om$ it exists 
$C_R>0$ such that
 \[
 \|\ee\|_{\hc(\Om_R)} + \|\hh\|_{V_{\hh,R}}  \leq C_R \|\ff\|_{\Vg^*}.
 \]
\end{theorem}
\begin{proof}
First of all, $\Vg \cap \hcg^{-1/2}(\Ga) = \Vg$ with equivalence of norms and we have equivalence 
between \eqref{pb:GIBCrotVol} and \eqref{pb:Surface} in the sense of Lemma \ref{le:EquiVolSurf}. 
 Let us prove that $\MtE_\Ga +\Z : \Vg \rightarrow \Vg^*$ is of Fredholm type. The operator
  $\MtE_\Ga : \hcg^{-1/2}(\Ga) \rightarrow \Vg^*$ is continuous and therefore 
is compact from $\Vg$ into $\Vg^*$ since we assumed that $\Vg$ is compactly embedded into $\hcg^{-1/2}(\Ga)$. 
Moreover, $\Z = \mc{T} + \mc{K}$ with  $\mc{T} : \Vg \rightarrow \Vg^*$ an isomorphism and $\mc{K} : \Vg \rightarrow 
\Vg^*$
a compact operator and consequently, $\MtE_\Ga +\Z$ is of Fredholm type with index $0$. 
Theorem \ref{th:MaxUnicite} allows to finish the proof.
\end{proof}

When $\Vg$ is not included into $\hcg^{-1/2}(\Ga)$, there is no real advantage in using the surface formulation to 
establish existence and uniqueness of the solution to the scattering problem and we have to impose some restrictions on 
the sign of the imaginary part of the boundary operator to obtain the following theorem.
\begin{theorem}
\label{th:div}
Let $\Z$ be an impedance operator such that Hypothesis \ref{hyp:uniciteMaxwell} is satisfied and such that it exists 
$c>0$ such that
\[
\Im \<\Z \uu,\uu\>_{\Vg^*,\Vg} >c\|\uu\|^2_{\Vg} \quad \forall \; \uu\in \Vg.
\]
Then for all $\ff \in \Vg^*$ problem \eqref{pb:GIBCrotVol} has a unique solution $(\ee,\hh) \in  \hcext(\Omext) \times 
V_{\hh} $ 
and for all ball $B_R$ that contains $\b \Om$ it exists 
$C_R>0$ such that
 \[
 \|\ee\|_{\hc(\Om_R)} + \|\hh\|_{V_{\hh,R}}  \leq C_R \|\ff\|_{\Vg^*}.
 \]
\end{theorem}
 \begin{proof}
 The proof of this result is a slight adaptation of the procedure presented in \cite[chapter 10]{Mon03} and is therefore postponed in appendix.
\end{proof}

\section{Well-posedness for second order surface differential operators}
In this section we will consider three different second order surface differential operators and we will see that each one of 
these operators requires a  different 
treatment. For the first two cases we will prove the Fredholm property of the surface formulation \eqref{pb:Surface} and use Theorem \ref{th:existgen} while 
 in the third case $\Vg$ is not a subspace of $\hcg^{-1/2}(\Ga)$ and we will make use of  Theorem \ref{th:div}.
\subsection{The case of $\Z =\Rotv\eta\Rot + \Grad\ga\Div   + \ld$}
We take $(\ld,\eta,\ga) \in (L^\infty(\Ga))^3$ and we define
\[
 \Z = \Rotv\eta\Rot + \Grad\ga\Div   + \ld
\]
which is a bounded and linear operator from  $\Vg = \hdg^1(\Ga)\cap \hcg^1(\Ga)$ into its dual. The space $\Vg$ is 
endowed with the norm 
\[
\|\cdot\|_{\Vg} := \|\cdot\|_{\hdg^1(\Ga)} +\|\cdot\|_{\hcg^1(\Ga)}
\]
and this space is nothing but  $\H^1_t(\Ga)$ since we have the algebraic relation 
\[
-\vec{\Delta}_\Ga = \Rotv\Rot - \Grad\Div
\]
where $\vec{\Delta}_\Ga$ is the vector Laplace Beltrami operator on $\Ga$.
As a consequence, the embedding of  $\Vg$ in $\hcg^{-1/2}(\Ga)$ and $\L^2_t(\Ga)$ is compact. Therefore we can use 
the surface formulation \eqref{pb:Surface} to 
prove that problem \eqref{pb:GIBCrotVol} is well-posed under the following sign assumptions on $\ld$, $\eta$ and $\ga$.
\begin{hypothese}
\label{hyp:MaxwelldirectGradRot}
The functions $(\ld,\eta,\ga) \in (L^\infty(\Ga))^3$ are such that
\[
 \Re(\ld) \geq 0 \ , \quad \Re(\eta) \geq 0,\quad \Re(\ga) \leq 0 \quad \text{a.e. on } \Ga,
\]
it exists $c>0$ such that
\[
 | \ga | \geq c \ , \quad |\eta| \geq c \quad \text{a.e. on } \Ga
\]
and the imaginary parts of $\ga$ and $\eta$ do not change sign on $\Ga$ and are of opposite sign.
\end{hypothese}
The following theorem is then a consequence of Theorem \ref{th:existgen}.
\begin{theorem}
 If $(\ld,\eta,\ga)$ satisfy Hypothesis \ref{hyp:MaxwelldirectGradRot} then for all $ \ff \in \Vg^*$ problem 
\eqref{pb:GIBCrotVol} with  $\Z = \Rotv\eta\Rot + \Grad\ga
\Div   + \ld$ has a unique solution $(\ee,\hh) $ and for all ball $B_R$ that contains $\b \Om$ it exists $C_R>0$ such that
 \[
 \|\ee\|_{\hc(\Om_R)} + \|\hh\|_{V_{\hh,R}}  \leq C_R \|\ff\|_{\Vg^*}.
 \]
\end{theorem}
\begin{proof}
Let us assume that $(\ld,\eta,\ga) \in (L^\infty(\Ga))^3$ satisfy Hypothesis \ref{hyp:MaxwelldirectGradRot}, then  the real 
part of $\Z$ is non negative and from 
Theorem \ref{th:MaxUnicite} we deduce that problem  \eqref{pb:GIBCrotVol} with  $\Z = \Rotv\eta\Rot + \Grad\ga\Div   + 
\ld$ has at most one solution. To prove 
existence we use the surface formulation \eqref{pb:Surface} which is equivalent to \eqref{pb:GIBCrotVol} since $\Vg 
\subset \hcg^{-1/2}(\Ga)$. 
Let us define  the bounded linear operators $T : \Vg \rightarrow \Vg^* $ and  $K : \Vg \rightarrow \Vg^*$ by
 \[
  \< T \vv,\ww\>_{\Vg^*,\Vg} := \int_\Ga \eta\ \Rot \vv\ \Rot\b{\ww}\ds - \int_\Ga \ga\ \Div \vv\ \Div\b{\ww} \ds + \int_\Ga \eta\ 
\vv\cdot\b{\ww} \ds,
 \]
\[
  \< K \vv,\ww\>_{\Vg^*,\Vg} :=  \int_\Ga (\ld-\eta)\ \vv\cdot\b{\ww} \ds
 \]
 for all $\vv$ and $\ww$ in $\Vg$ and then $\Z = T + K$.  We recall that for any complex number $z = a+ ib \in \bC$ we 
have
 \begin{equation}
 \label{eq:modulecomplexe}
 |z| \geq \frac{|a| + |b|}{\sqrt{2}}
 \end{equation}
 and since $(\ld,\eta,\ga)$ satisfy Hypothesis  \ref{hyp:MaxwelldirectGradRot}, this last inequality implies that the operator 
$T$ is coercive on $\Vg$ i.e. it exists 
$C>0$ such that
 \[
| \<T\uu,\uu\>_{\Vg^*,\Vg}| \geq C \|\uu\|^2_{\Vg} \quad \forall \ \uu \in \Vg.
 \]
Moreover, since the embeddings of $\Vg$ into  $\L^2_t(\Ga)$ and $\hcg^{-1/2}(\Ga)$ are compact, we deduce that $K : 
\Vg \rightarrow \Vg^*$ and $\MtE_\Ga : \Vg 
\rightarrow \Vg^*$ are compact operators. Then, Theorem \ref{th:existgen}  concludes the proof.
\end{proof}
\begin{remark}
Using formulation \eqref{pb:Surface} instead of formulation \eqref{pb:GIBCBR} to prove existence and uniqueness of the 
solution to the scattering problem allows 
us to treat the case of  coefficients $\eta$ and $\ga$ with positive and negative imaginary parts respectively. This could 
not be achieved with standard 
variational arguments on the variational formulation \eqref{eq:varformglo} associated with the volume problem \eqref{pb:GIBCBR}.
\end{remark}

\subsection{The case of  $\Z = \Rotv\eta\Rot + \ld$}
\label{SecCurl}
Let $\ld$ and $\eta$ be two $L^\infty(\Ga)$ functions and let us define
\[
 \Z = \Rotv\eta\Rot + \ld
\]
which is bounded and continuous from $\Vg = \hcg^1(\Ga)$ into its dual. We assume that $\ld$ and $\eta$ satisfy the 
following sign hypothesis.
\begin{hypothese}
\label{hyp:Maxwelldirect}
The functions $(\ld,\eta) \in (L^\infty(\Ga))^2$ are such that
\[
 \Re(\ld) \geq 0 \ , \quad \Re(\eta) \geq 0 \quad \text{a.e. on } \Ga,
\]
it exists $c>0$ such that
\[
 | \ld | \geq c \ , \quad |\eta| \geq c \quad \text{a.e. on } \Ga
\]
and the imaginary parts of  $\ld$ and  $\eta$ do not change sign on $\Ga$.
\end{hypothese}
First of all, if the imaginary parts of $\ld$ and $\eta$ are of the same sign then the situation is very similar to the one in 
the previous section. Actually, $\Vg$ is 
compactly embedded into $\hcg^{-1/2}(\Ga)$ and we can use the surface formulation \eqref{pb:Surface} to prove that 
problem \eqref{pb:GIBCrotVol} is well-posed. 
Indeed, in this case $\Z : \Vg \rightarrow \Vg^*$ is coercive and since $
\MtE_\Ga : \Vg \rightarrow \Vg^*$ is compact 
we deduce well-posedness from the uniqueness Theorem \ref{th:MaxUnicite}.

If this is not the case, that is if the imaginary parts of $\ld$ and $\eta$ are of opposite sign then we have to be much more 
careful to prove existence of a solution to 
\eqref{pb:GIBCrotVol} (uniqueness is ensured by Theorem \ref{th:MaxUnicite}).
As mentioned in the introduction, this happens for example in the case where $\Z$ models a thin layer of metal.  Actually, the $\ld $ part of the impedance operator has to be treated as a compact perturbation of the $\Rotv \eta \Rot$  operator but it not true. Actually, similarly to  the 
volume spaces, $\hcg^1(\Ga)$ is not compactly embedded into $\L^2_t(\Ga)$ since, for example, for all $p\in H^1(\Ga)$ 
we have
\[
\Rot\nabla_\Ga p=0.
\] 
Nevertheless, we prove in what follows that $ \Z + \MtE_\Ga :  \Vg\rightarrow \Vg^*$ is an isomorphism by using a 
Helmholtz' decomposition of $\Vg$. Before giving 
the actual decomposition we need to introduce some additional notations. For  any $\ff \in \Vg^*$, let us define the 
sesquilinear form $a_\Ga$ on $\Vg\times \Vg$ 
and the anti-linear form $l_\Ga$ on $\Vg$ by 
\[
a_\Ga(\uu,\vv) := \int_\Ga( \eta\Rot \uu \  \Rot \b{\vv} +\ld \uu\cdot\b{\vv} )\ds+ \<\MtE_{\Ga}(\uu) ,\vv\>_{\hdg^{-1/2}(\Ga),
\hcg^{-1/2}(\Ga)} \quad \forall \  (\uu,\vv) \in 
(\Vg)^2,
\]
\[
l_\Ga(\vv) := \<\ff,\vv\>_{\Vg^*,\Vg} \quad \forall \ \vv \in \Vg.
\] 
Then, $\uu_\Ga \in \Vg$ solves $\eqref{pb:Surface}$ if and only if
\[
a_\Ga(\uu_\Ga,\vv) = l_\Ga(\vv)
\]
for all $\vv\in \Vg$.

Let us define
\[
 \mathring H^1(\Ga) :=\left\{ p \in H^1(\Ga)\,|\, \int_{\Ga}p\ds =0\right\}
\]
the space of  $H^1(\Ga)$ functions with zero mean on  $\Ga$ endowed with the $H^1(\Ga)$ norm and 
\[
 X := \left\{ \vv \in \Vg \,|\, \int_{\Ga}\ld   \vv\cdot\nabla_\Ga\overline{\xi}\ds + \<\MtE_\Ga(\vv),\nabla_\Ga \xi\>_{\hdg^{-1/2}
(\Ga),\hcg^{-1/2}(\Ga)} =0 
 \quad \forall\  \xi \in  \mathring H^1(\Ga) \right\}
\]
endowed with the $\hcg^1(\Ga)$ norm. These two spaces are Hilbert spaces. We prove in Lemma \ref{le:HelmSurf} 
that
\[
\Vg =  \Grad \mathring H^1(\Ga) \oplus X
\]
and in Lemma \ref{le:CompactSurf} that the embedding of $X$ in $\L^2_t(\Ga)$ is compact. To this end, let us first 
introduce the operator $A_S :  \mathring 
H^1(\Ga)  \rightarrow  \mathring H^1(\Ga) $ defined by
 \begin{align*}
  (A_S p,\xi)_{H^1(\Ga)} &:=a_\Ga(\nabla_\Ga p,\nabla_\Ga\xi) \\
  &= \int_\Ga  \ld \nabla_\Ga p\cdot\nabla_\Ga\overline{\xi}\ds + \<\MtE_\Ga(\nabla_\Ga p),\nabla_\Ga \xi\>_{\hdg^{-1/2}
(\Ga),\hcg^{-1/2}(\Ga)}   \end{align*}
for all $ (p,\xi) \in  (\mathring H^1(\Ga))^2$. According to the next lemma, $A_S$ is an isomorphism.
 \begin{lemma}
 \label{le:ASiso}
If $\ld$ satisfies Hypothesis \ref{hyp:Maxwelldirect} then $A_S :  \mathring H^1(\Ga)  \rightarrow  \mathring H^1(\Ga) $ is 
an isomorphism.
 \end{lemma}
\begin{proof}
Let $\ld \in L^\infty(\Ga)$ be such that Hypothesis \ref{hyp:Maxwelldirect}  is satisfied. Let $C_S$ and $K_S$ be the 
two bounded operators from  $ \mathring H^1(\Ga) $ 
into $ \mathring H^1(\Ga) $ defined by
\[
 (C_S p,\xi)_{H^1(\Ga)} = \int_\Ga  \ld (\nabla_\Ga p\cdot\nabla_\Ga\overline{\xi} + p\overline{\xi})\ds
 \quad\forall\, (p,\xi) \in  (\mathring H^1(\Ga))^2,
\]
\[
 (K_S p,\xi)_{H^1(\Ga)} =-\int_\Ga  \ld \ p\ \overline{\xi}\ds + \<\MtE_\Ga(\nabla_\Ga p),\nabla_\Ga \xi\>_{\hdg^{-1/2}(\Ga),
\hcg^{-1/2}(\Ga)}
 \quad\forall\, (p,\xi) \in  (\mathring H^1(\Ga) )^2,
\]
then $A_S=C_S+K_S$. First of all, from \eqref{eq:modulecomplexe} and since the imaginary part of  $\ld$ does not 
change sign, we have for all $p \in  \mathring 
H^1(\Ga) $
\[
 | (C_S p,p)_{H^1(\Ga)} | \geq \frac{1}{\sqrt{2}}\int_\Ga  [\Re(\ld)+|\Im(\ld)|]( |\nabla_\Ga p|^2 + |p|^2)\ds \geq  \frac{c}
{\sqrt{2}} \|p\|_{H^1(\Ga)}
\]
where $c$ is the lower bound on the modulus of $\ld$ and $c>0$ from Hypothesis \ref{hyp:Maxwelldirect}. Hence $C_S$ 
is an isomorphism from Lax-Milgram 
Lemma.

We prove that $K_S : \mathring H^1(\Ga)  \rightarrow  \mathring H^1(\Ga) $ is compact. Let $(p_n)_n$ be a bounded 
sequence of $ \mathring H^1(\Ga) $, let us 
prove that we can extract from $(K_S p_n)_n$ a subsequence that converges in $ \mathring H^1(\Ga) $. 
From the definition of $K_S$ and using the continuity of $\MtE_\Ga : \hcg^{-1/2}(\Ga) \rightarrow \hdg^{-1/2}(\Ga) $ we 
deduce that it exists a constant $C>0$ such 
that for all  $n\in \bN$ we have
\[
 \|K_S p_n\|^2_{H^1(\Ga)} \leq \|\ld\|_{L^\infty(\Ga)} \|p_n\|_{L^2(\Ga)}\|K_S p_n\|_{L^2(\Ga)} + C \|\nabla_\Ga p_n\|
_{\hcg^{-1/2}(\Ga)}\|\nabla_\Ga(K_S p_n)\|
_{\hcg^{-1/2}(\Ga)}.
\]
But since $\Rot(\nabla_\Ga p_n)=0$ we deduce that $\|\nabla_\Ga p_n\|_{\hcg^{-1/2}(\Ga)} = \|\nabla_\Ga p_n\|
_{\H_t^{-1/2}(\Ga)}$. Similarly, we obtain that  it 
exists $C>0$ such that for all  $n\in \bN$:
\[
    \|\nabla_\Ga(K_S p_n)\|_{\hcg^{-1/2}(\Ga)}=  \|\nabla_\Ga(K_S p_n)\|_{\H_t^{-1/2}(\Ga)}\leq C \|K_S p_n\|_{H^{1/2}(\Ga)}.                                                                                                                                                                                                                                                           
\]
Therefore, we obtain that it exists $C>0$ such that for all  $n\in \bN$:
\begin{equation}
\label{eq:borneKpn}
 \|K_S p_n\|_{H^1(\Ga)} \leq C\left( \|p_n\|_{L^2(\Ga)} +  \|\nabla_\Ga p_n\|_{\H^{-1/2}_t(\Ga)}\right) \leq C  \|p_n\|
_{H^{1/2}(\Ga)}.
\end{equation}
We recall that the sequence $(p_n)_n$ is bounded in $H^1(\Ga)$ and therefore one can extract from $(p_n)_n$ a 
subsequence still denoted by $(p_n)_n$ that is 
of Cauchy type in $H^{1/2}(\Ga)$. This observation together with inequality \eqref{eq:borneKpn} implies that $K_S$ is 
compact and therefore that $A_S= C_S + K_S$ is of 
Fredholm type with index $0$ since $C_S$ is an isomorphism.

To conclude the proof, let us prove that $A_S$ is injective. We take 
 $p \in \mathring H^1(\Ga)$ such that $A_Sp=0$. We then have
\begin{equation}
\label{EqInj}
 \int_\Ga  \ld |\nabla_\Ga p|^2\ds + \<\MtE_\Ga(\nabla_\Ga p),\nabla_\Ga p\>_{\hdg^{-1/2}(\Ga),\hcg^{-1/2}(\Ga)} =0.
\end{equation}
Let $(\ee,\hh)$ be the unique solution to \eqref{pb:MaxwellDir} with  $\vv= \nabla_\Ga p$ 
on $\Gamma$. As in the proof of Theorem \ref{th:MaxUnicite} we get by integration by part that
\[
\Re\< \nn\times \ee, \hh_T\>_{\hdg^{-1/2}(\Ga),\hcg^{-1/2}(\Ga)}= \Re\left(\int_{\partial B_R} \hat{x} \times \ee \cdot\b{\hh} 
\ds\right)
\]
for all ball $B_R $ that is such that $\b{\Om} \subset B_R$. Since $\Re(\ld) \geq 0 $, this last inequality together with Rellich's 
Lemma and the unique continuation 
principle implies that 
 $\ee=\hh=0$ in $\Omext$ and as a consequence $\nabla_\Ga p=0$. 
 Since $p$ has a zero mean on $\Ga$, this implies $p=0$ which concludes the proof.
\end{proof}
We make use of the isomorphism $A_S$ to prove the following Helmholtz' decomposition.
\begin{lemma}
\label{le:HelmSurf}
If $\ld$ satisfies Hypothesis \ref{hyp:Maxwelldirect} then $\Vg$ writes as the direct sum 
of $\nabla_\Ga  \mathring H^1(\Ga) $ and  $X$:
 \[
\Vg = \nabla_\Ga  \mathring H^1(\Ga) \oplus X,
 \]
and there exists $C>0$ such that
\[
 \|\ww\|_{\Vg}+\|\nabla_\Ga p\|_{\Vg} \leq C\|\nabla_\Ga p+ \ww\|_{\Vg}
\]
for all $\ww \in X$ and $p\in  \mathring H^1(\Ga) $.
\end{lemma}
\begin{proof}
Let us  take 
$\uu  \in \Vg$,  and let us define $F$ as being the unique function of $ \mathring H^1(\Ga) $ that satisfies
\[
(F,\xi)_{H^1(\Ga)} = \int_\Ga  \ld\uu \cdot\nabla_\Ga\overline{\xi}\ds + \<\MtE_\Ga(\uu),\nabla_\Ga \xi\>_{\hdg^{-1/2}(\Ga),
\hcg^{-1/2}(\Ga)}\quad \forall \ \xi \in  
\mathring H^1(\Ga) .
\]
Since $A_S : \mathring H^1(\Ga) \rightarrow \mathring H^1(\Ga)$ is an isomorphism (Lemma \ref{le:ASiso}), it exists a 
unique  $p \in  \mathring H^1(\Ga) $ such 
that $A_Sp = F$ and it exists $C>0$  such that
\begin{equation}
\label{eq:contp}
\|p\|_{H^1(\Ga)} \leq C \|\uu\|_{\Vg}.
\end{equation}
Let us define $\ww:=\uu-\nabla_\Ga p$,  from the definition of  $A_S$ and $p$, we have
\[
\int_\Ga \ld \ww \cdot\nabla_\Ga\overline{\xi}\ds + \<\MtE_\Ga(\ww),\nabla_\Ga \xi\>_{\hdg^{-1/2}(\Ga),\hcg^{-1/2}(\Ga)} = 
(F- A_Sp,\xi)_{H^1(\Ga)} = 0\quad \forall \ 
\xi \in  \mathring H^1(\Ga) 
\]
whence, $\ww \in X$ and by \eqref{eq:contp} we have the following continuity relation
\[
 \|\ww\|_{\Vg}+\|\nabla_\Ga p\|_{\Vg}  \leq  \|\uu\|_{\Vg} + 2 \| \nabla_\Ga p\|_{\Vg} \leq (2C +1) \|\uu\|_{\Vg}
\]
for $C>0$.  We have then proven that for any  $\uu \in \Vg$  it exists 
$p\in  \mathring H^1(\Ga) $ and $\ww \in X$ such that $\uu=\nabla_\Ga p +\ww$. 
We now only have to prove that the sum between $\mathring H^1(\Ga) $ and $X$ is direct. For $\uu=\nabla_\Ga p \in X 
\cap \nabla_\Ga  \mathring H^1(\Ga) $ we 
have $A_Sp=0$ since $\uu \in X$. Hence $p=\uu=0$ since $A_S$ is injective. This concludes the proof.
\end{proof}

In order to prove the compact embedding of $X$ into $\L^2_t(\Ga)$ we need some classical regularity properties for 
Maxwell's equations that we recall in Theorem 
\ref{th:RegTr} (see \cite{Cos90} for a proof of this result in Lipschitz domains) as well as the compactness result 
established in Lemma  \ref{le:DivCurlCompact}.
\begin{theorem}
\label{th:RegTr}
Let $\mathcal{O} \in \bR^3$ be a bounded simply connected domain with  Lipschitz boundary. Let us assume that $\uu \in 
\hc(\mathcal{O})$ is such that $\div (\uu) 
\in L^2(\mathcal{O})$ and  $\uu_T \in \L^2_t(\partial\mathcal{O})$, then it exists $C>0$ such that
\[
   \|\uu  \|_{(H^{1/2}(\mathcal{O}))^3} \leq C[\|\uu\|_{\hc(\mathcal{O})} + \|\div(\uu)\|_{(L^2(\mathcal{O}))^3} + \|\uu_T\|
_{\L^2_t(\partial \mathcal{O})}]
\]
and
\[
   \|\uu\cdot\nn  \|_{L^2(\partial O)} \leq C[\|\uu\|_{\hc(\mathcal{O})} + \|\div(\uu)\|_{(L^2(\mathcal{O}))^3} + \|\uu_T\|
_{\L^2_t(\partial \mathcal{O})}].
\]
\end{theorem}
\begin{lemma}
\label{le:DivCurlCompact}
Let $\sigma$ be a $L^\infty(\Ga)$ function such that  $|\sigma(x)|>c>0$ for almost all $x \in \Ga$ and that is such  that its 
real and imaginary parts do not change 
sign. Then, the space $\H^1_{t,\sigma} (\Ga):= \{ \uu \in \Vg \,|\,\Div(\sigma \uu) \in L^2(\Ga)\}$ is compactly embedded 
into $\L_t^2(\Ga)$.
\end{lemma}
\begin{proof}
Let $(\uu_n)_n$ be a bounded sequence in  $\H^1_{t,\sigma}(\Ga)$, then there exists $C>0$ such that for all $n\in \bN$ 
we have
\[
\| \uu_n\|_{\L_t^2(\Ga)}\leq C\,,\quad \|\Div(\sigma \uu_n)\|_{L^2(\Ga)} \leq C\quad \text{and} \quad \|\Rot( \uu_n)\|
_{L^2(\Ga)}\leq C.
\]
We define $\vp_n$ as being the unique function in  $ \mathring H^1(\Ga)$ that satisfies 
\begin{equation}
\label{eq:DivGa}
\Div(\sigma\nabla_{\Ga} \vp_n) =\Div (\sigma \uu_{n}) ,
\end{equation}
then $\sigma(\uu_n - \nabla_{\Ga} \vp_n)$ has a vanishing surface divergence and is in  $ \L^2_t(\Ga)$. As a 
consequence, it exists  $v_n \in  \mathring H^1(\Ga) $ 
such that $\Rotv v_n = \sigma(\uu_n - \nabla_{\Ga} \vp_n)$ and then $\uu_n = \nabla_\Ga \vp_n + \frac{1}{\sigma} \Rotv 
v_n$. We now prove that we can extract a  
subsequence from $(\Rotv v_n)_{n\in \bN}$ and from  $(\nabla_\Ga \vp_n)_{n\in \bN}$ that converge in $\L_t^2(\Ga)$. 

First of all, since  (\ref{eq:DivGa}) has a unique solution in  $ \mathring H^1(\Ga) $ that depends continuously on the 
right-hand side, it exists $C>0$ such that $\|
\vp_n\|_{H^1(\Ga)} \leq C \|\Div (\sigma \uu_{n})\|_{L^2(\Ga)}$. The sequence $(\vp_n)_{n\in \bN}$ is in particular 
bounded in  $H^1(\Ga)$ therefore we can extract 
from it a subsequence still denoted by $(\vp_n)_{n\in \bN}$ that converges in $L^2(\Ga)$. We prove next that it is of 
Cauchy type in $H^1(\Ga)$. Let us define $
\vp_{nm} := \vp_n - \vp_{m}$ and  $f_{nm} := \Div (\sigma \uu_{n}) - \Div (\sigma \uu_{m})$, then there exists $C>0$ such that
\[
 \|\nabla_\Ga \vp_{nm}\|^2_{\L_t^2(\Ga)} \leq C \left|\int_{\Ga} \sigma \nabla_\Ga \vp_{nm}\cdot\nabla_\Ga 
\overline{\vp_{nm}}\ds \right| =C\left|\int_{\Ga} f_{nm} 
\overline{\vp_{nm}} \ds \right|.
\]
Since $f_{nm}$ is bounded  in $L^2(\Ga)$ and $(\vp_{n})_{n\in \bN}$ is a Cauchy sequence in $L^2(\Ga)$, we obtain that 
$(\nabla_\Ga \vp_{n})_{n\in \bN}$ is a 
Cauchy sequence in $L^2(\Ga)$ whence $(\vp_n)_{n\in \bN}$ converges in $H^1(\Ga)$. 

Concerning $(v_n)_{n\in \bN}$ we proceed in a similar way. First of all, it exists  $C>0$ such that $\|\Rotv v_n\|
_{\L^2_t(\Ga)} = \|\nabla_\Ga v_n\|_{\L^2_t(\Ga)} \geq 
C \|v_n\|_{H^1(\Ga)}$ since $\Ga$ is $C^{1,1}$ (this is still true for a Lipschitz boundary). But, we recall that $\Rotv v_n =
\sigma(\uu_n - \nabla_{\Ga} \vp_n) $,
therefore it is a bounded sequence in $\L^2_t(\Ga)$. From the compact embedding of  $H^1(\Ga)$ in $L^2(\Ga)$,  we 
deduce that we can extract from $(v_n)_{n\in 
\bN}$ a subsequence still denoted  $(v_n)_{n\in \bN}$ that converges in $L^2(\Ga)$. As previously, we conclude by 
proving that $(\Rotv v_n)_{n\in \bN}$ is a 
Cauchy sequence in $\L^2_t(\Ga)$. Let us denote  $f_{nm} := \Rot (\uu_n)-\Rot(\uu_m)$, there exists $C>0$ such that
\[
 \|\Rotv v_{nm}\|^2_{(\L_t^2(\Ga))^3} \leq C \left|\int_{\Ga} \sigma^{-1} \Rotv v_{nm}\cdot\Rotv \overline{v_{nm}}  \ds \right| 
= C\left|\int_{\Ga} f_{nm} \overline{v_{nm}} ,
\ds \right|
\]
whence $\Rotv v_{n}$ is a Cauchy sequence in  $\L^2_t(\Ga)$ and it converges in $\L^2_t(\Ga)$. 
This concludes the proof since we have proven that one can extract a sequence of $(\uu_n)_{n\in \bN}$ that converges in 
$\L^2_t(\Ga)$. 
\end{proof}
The following lemma definitely justifies the use of the Helmholtz' decomposition introduced in Lemma \ref{le:HelmSurf}.
\begin{lemma}
\label{le:CompactSurf}
If $\ld\in L^\infty(\Ga)$ satisfies Hypothesis \ref{hyp:Maxwelldirect} then the embedding of $X$  into $\L^2_t(\Ga)$ is 
compact.
\end{lemma}
\begin{proof}
Let $(\uu_n)_n$ be a bounded sequence of  $X$, then it exists $C>0$ such that for all  $n \in \bN$
 \[
  \|\uu_n\|_{\Vg} \leq C
 \]
and since $\uu_n \in X$, we also have
 \[
  \Div(\ld \uu_n) =- \Div(\MtE_\Ga (\uu_n))
 \]
 in the sense of distributions. 
 We define $(\ee_n,\hh_n)$ as being the unique solution to \eqref{pb:MaxwellDir} with $\vv=\uu_n$ on  $\Ga$.  By using 
\eqref{eq:rotequality} we have that
\[
 \Div(\ld \uu_n) = - \Div(\nn\times\ee_n)= \nn\cdot\rot\ee_n =  i\om \nn\cdot\hh_n.
\]
Whence, since $\hh_{n,T} = \uu_n$, it exists $C>0$ such that
\[
\|\hh_n\|_{\hc(\Om)}\leq C \quad \text{and} \quad \|\hh_{n,T}\|_{\L^2_t(\Ga)}\leq C.
\] 
From Theorem  \ref{th:RegTr} we deduce that it exists $C>0$ such that for all $n \in \bN$
\[
\|\Div(\ld \uu_n)\|_{ L^2(\Ga)}=k \|\nn\cdot\hh_n\|_{ L^2(\Ga)} \leq C.
\]
Lemma \ref{le:DivCurlCompact} proves then that we can extract a sequence of  $(\uu_n)_n$ that converges in $
\L^2_t(\Ga)$ which finishes the proof.
\end{proof}

We now conclude the study of well-posedness of problem \eqref{pb:GIBCrotVol} for $\Z = \Rotv\eta\Rot + \ld$.
\begin{theorem}
\label{th:existrot}
Let $(\ld,\eta) \in (L^\infty(\Ga))^2$ be such that Hypothesis \ref{hyp:Maxwelldirect} is satisfied. 
Then for all $\ff \in \Vg^*$ problem \eqref{pb:GIBCrotVol} with $\Z = \Rotv\eta\Rot + \ld$ has a unique solution $(\ee,\hh) $ 
and for all ball $B_R$ that contains $\b 
\Om$ it exists $C_R>0$ such that
 \[
 \|\ee\|_{\hc(\Om_R)} + \|\hh\|_{V_{\hh,R}}  \leq C_R \|\ff\|_{\Vg^*}.
 \]
 \end{theorem}
\begin{proof}
 We take $\ff\in \Vg^*$ and $(\ld,\eta) \in (L^\infty(\Ga))^2$ such that Hypothesis \ref{hyp:Maxwelldirect} is satisfied.
 Since $\Vg  = \hcg^1(\Ga) \subset \hcg^{-1/2}(\Ga)$, we know from Lemma \ref{le:EquiVolSurf} that problem 
\eqref{pb:GIBCrotVol} is equivalent to 
problem \eqref{pb:Surface}. As a consequence, it is sufficient to prove that \eqref{pb:Surface} is well-posed.
Theorem \ref{th:MaxUnicite}, gives uniqueness, we only have to prove existence.
We look for a solution $\uu$ that writes  $\uu = \uu_0+\Grad p$ with $\uu_0 \in X$ and  $p\in  \mathring H^1(\Ga) $. The 
function $\uu$ has to satisfiy
 \[
  a_\Ga(\uu,\vv)= l_\Ga(\vv) \quad \text{for all } \vv\in \Vg
 \]
which if we use test functions being gradients of functions of $ \mathring H^1(\Ga) $ implies that  $p$ has to satisfy
\begin{equation}
\label{eq:definitionp}
 (A_S p,\xi) = l_\Ga(\Grad \xi)  \quad \text{for all } \xi \in  \mathring H^1(\Ga).
\end{equation}
Let us recall that $A_S$ is an isomorphism of $ \mathring H^1(\Ga)$, therefore \eqref{eq:definitionp} has a unique solution  $p \in  \mathring 
H^1(\Ga) $. If now we use test functions in  $X$, 
we obtain that $\uu_0$ has to satisfy
\begin{equation}
\label{eq:definitionu0}
 a_\Ga(\uu_0,\vv_0)= l_\Ga(\vv_0) -a_\Ga(\Grad p,\vv_0) \quad \text{for all } \vv_0\in X.
\end{equation}
Let us prove that \eqref{eq:definitionu0} has a unique solution in the Hilbert space $X$. We define $C_X : X \rightarrow X
$ and $K_X :X \rightarrow X$ the bounded 
and linear operators that satisfy
\[
 (C_X \vv,\ww)_{\Vg} = \int_\Ga \eta\ (\Rot \vv\ \Rot \b{\ww} +  \vv \cdot \b{\ww}) \ds
\]
\[
 (K_X \vv,\ww)_{\Vg} = \int_\Ga (-\eta +\ld) \vv\ \overline{\ww} \ds +\<\MtE_\Ga(\vv),\ww\>_{\hdg^{-1/2}(\Ga),\hcg^{-1/2}
(\Ga)} 
\]
for all $\vv,\ww \in X$.
With these definitions $((C_X + K_X)\vv,\ww)_{\Vg}=a_\Ga(\vv,\ww)$  for all $\vv,\ww \in X$. 
From Theorem \ref{th:MaxUnicite} we know that  $C_X+K_X$ is injective, let us prove that 
$C_X + K_X$ is a Fredholm type 
operator of index $0$. First of all, $C_X$ is coercive since $\eta$ satisfies Hypothesis \ref{hyp:Maxwelldirect}. Moreover, 
$\MtE_\Ga : \Vg \rightarrow \Vg^*$ is 
compact  and since the injection of $X$ in $\L^2_t(\Ga)$ is compact  (Lemma \ref{le:CompactSurf}) we deduce 
that  $K_X$ is a compact operator. This 
guaranties well-posedness of 
 \eqref{eq:definitionu0} which has a unique solution $\uu_0\in X$ that depends continuously on $p$ and $\ff$. 
 To conclude, we  built a function  $\uu = \uu_0 + \Grad p$ that solves \eqref{pb:Surface}. We obtain the continuous 
dependence of $\uu$ with respect to $\ff$ by using 
Lemma \ref{le:HelmSurf} together with the fact that $A_S : \mathring H^1(\Om) \rightarrow  \mathring H^1(\Om)$ and 
$C_X+K_X  : X \rightarrow X$ are 
isomorphisms.
\end{proof}

 \subsection{The case of  $\Z= \Grad\ga\Div   + \ld$}
 \label{sec:div}
 We conclude this serie of examples with a third one for which we cannot use the surface formulation \eqref{pb:Surface}.
Let us consider
 \[
 \Z = \Grad\ga\Div   + \ld
 \]
for $(\ld,\ga)$ two functions of  $L^\infty(\Ga)$. This operator is linear and continuous from $\Vg:=\hdg^1(\Ga) $ into its 
dual and we cannot use the formulation 
\eqref{pb:Surface} since $\Vg$ is not included into $\hcg^{-1/2}(\Ga)$ in this case. Nevertheless, we show that under 
appropriate sign assumptions on $\ld$ and $
\ga$ we can apply Theorem \ref{th:div}.
\begin{hypothese}
\label{hyp:Maxwelldirectdivgrad}
The functions $(\ld,\ga) \in (L^\infty(\Ga))^2$ are such that
\[
 \Re(\ld) \geq 0 \ , \quad \Re(\ga) \leq 0 \quad \text{a.e. on } \Ga,
\]
and it exists $c>0$ such that
\[
 \Im( \ld ) \geq c \ , \quad \Im(\ga) \leq -c \quad \text{a.e. on } \Ga.
\]
\end{hypothese}
Under this restrictive sign assumptions (compare to the two previous examples) $\Z$ satisfies assumptions of Theorem \ref{th:div} and we have the following result. 
\begin{theorem}
Let $(\ld,\ga) \in (L^\infty(\Ga))^2$ be such that Hypothesis  \ref{hyp:Maxwelldirectdivgrad} is satisfied. Then for all $\ff \in 
\Vg^*$ problem  \eqref{pb:GIBCrotVol}  
has a unique solution $(\ee,\hh)$ and for all ball $B_R$ that contains $\b \Om$ it exists $C_R>0$ such that
 \[
  \|\ee^s\|_{\hc(\Om_R)}+ \|\hh^s\|_{V_{\hh,R}} \leq C \|\ff\|_{\Vg^*}.
 \]
\end{theorem}

\begin{remark}
We can generalise further the results of this section to the case of a vanishing functions $\ld$ and $\ga$ on $\Ga$. In this case we 
use 
\[
\Vg := \left \{\vv\in \hcg^{-1/2}(\Ga) \phantom{\int_\Ga}\right. \left| \ \int_\Ga |\ld| |\vv|^2 + |\ga| |\Div \vv|^2 \ ds <+\infty \right\}
\]
endowed with the norm
\[
\| \vv\|^2_{\Vg} := \|\vv\|^2_{\hcg^{-1/2}(\Ga)} + \int_\Ga |\ld| |\vv|^2 + |\ga| |\Div \vv|^2 \ ds . 
\]
Existence and uniqueness is then ensured as soon as 
\[
 \Re(\ld) \geq 0 \ , \quad \Re(\ga) \leq 0 \quad \text{a.e. on } \Ga,
\]
and 
\[
 \Im( \ld ) \geq 0 \ , \quad \Im(\ga) \leq 0 \quad \text{a.e. on } \Ga.
\]
The question of existence of a solution when $\ld$ or $\ga$ have a negative 
imaginary part cannot be treated in this way 
and to the knowledge of the author is still open. 
\end{remark}

\section*{Appendix - proof of Theorem \ref{th:div}}
First of all, uniqueness holds from Theorem \ref{th:MaxUnicite}. To prove existence we adapt the procedure presented in 
\cite[chapter 10]{Mon03} in the case of a Dirichlet type boundary condition to the volume 
formulation \eqref{pb:GIBCBR}. We do not give a precise proof but we only highlight the main steps since it is rather 
classical.
As stated in section \ref{sec:volbouded}, the electromagnetic field $(\ee,\hh)$ solves \eqref{pb:GIBCBR} if and only if $\hh$ solves the 
variational formulation \eqref{eq:varformglo}. 
To study it we introduce a Helmholtz' decomposition for $\hc(\Om_R)$ in order to handle the $L^2(\Om_R)$ contribution 
which is not a compact perturbation of the principal part. Let us introduce the following Hilbert spaces
\[
 H_0^1(\Om_R) := \left\{p \in H^1(\Om_R)\ |   \   p=0 \text{ on } \Gamma\right\}
\]
and
\begin{align*}
 X_R &:=  \left\{ \uu \in V_{\hh,R} \  \phantom{\int_p}  \right | \\ 
 & \left. \int_{\Om_R} \om^2 \uu \cdot \nabla \overline{\xi} \dx  + i\om \< \MtE_R(\hat{x}\times\uu),\nabla_{\partial B_R} \xi 
\>_{\H^{-1/2}_{\div_{\partial B_R}}(\partial B_R),\H^{-1/2}_{\text{curl}_{\partial B_R}}(\partial B_R)}
=0 \ \forall \xi \in  H_0^1(\Om_R) \right\}.
\end{align*}
Let us define the operator $A_R :  H_0^1(\Om_R) \rightarrow  H_0^1(\Om_R)$ characterised by  
\[
(A_R p, \xi)_{H^1(\Om_R)} :=  \int_{\Om_R}  \om^2 \nabla p \cdot \nabla \overline{\xi} \dx  + i\om \< \MtE_R(\hat{x}\times
\nabla_{\partial B_R} p),\nabla_{\partial B_R} \xi \>_{\H^{-1/2}_{\div_{\partial B_R}}(\partial B_R),\H^{-1/2}_{\text{curl}
_{\partial B_R}}(\partial B_R)}
\]
for all $p, \xi$ in $H^1_0(\Om_R)$. Due to the symmetry of Maxwell's equations, the Magnetic-to-Electric map $\MtE_R$ 
is equal to $-G_R$ where $G_R$ is the Electric-to-Magnetic map which maps $\ee_T$ to $\hat x\times  \hh$ where $(\ee,\hh)$ 
solves Maxwell's equations outside $B_R$ together with the Silver-Mueller radiation condition. As a consequence, we 
can use the results of \cite{Mon03} Lemma 9.23 and 9.24 that state that it exists $\wt{\MtE}_R : \H^{-1/2}_{\div_{\partial 
B_R}}(\partial B_R) \rightarrow \H^{-1/2}_{\div_{\partial B_R}}(\partial B_R)$ such that for all $\uu \in  \H^{-1/2}
_{\div_{\partial B_R}}(\partial B_R)$ we have
\[
\<\wt{\MtE}_R \uu,\uu\times\hat x\>_{\H^{-1/2}_{\div_{\partial B_R}}(\partial B_R),\H^{-1/2}_{\text{curl}_{\partial B_R}}
(\partial B_R)} \geq c \|\uu\|^2_{\H^{-1/2}_{\div_{\partial B_R}}(\partial B_R)}
\]
for $c>0$ and  $\MtE_R + i\om \wt{\MtE}_R : \H^{-1/2}_{\div_{\partial B_R}}(\partial B_R) \rightarrow \H^{-1/2}
_{\div_{\partial B_R}}(\partial B_R)$ is a compact operator. We deduce that $A_R$ is an isomorphism and similarly to  the 
proof of Lemma \ref{le:HelmSurf} we obtain the following Helmholtz' decomposition
\[
V_{\hh,R} = X_R \oplus \nabla  H^1_0(\Om_R).
\]
Moreover, $X_R$ is compactly embedded into  $(L^2(\Om_R))^3$ (see the proof of \cite[Lemma 10.4]{Mon03}).
We also remark that from the sign assumption on the imaginary part of $\Z$,  it exists $C>0$ such that
\begin{equation}
\label{eq:CoercZweak}
\left|\int_{\Om_R} |\rot\ \uu|^2 + |\uu|^2\ dx - i\om \< \Z \uu_T,\uu_T \>_{\Vg^*,\Vg}\right| \geq C \|\uu\|^2_{V_{\hh,R}}
\end{equation}
for all $\uu \in V_{\hh,R}$. Finally, let us recall the result of Lemma 10.5 in \cite{Mon03} that states that $S_R$ can be 
decomposed as $S_R = S_1 +S_2$ where $S_1 : \H^{-1/2}_{\div_{\partial B_R}}(\partial B_R) \rightarrow \H^{-1/2}
_{\div_{\partial B_R}}(\partial B_R)$ has a positive imaginary part and  $S_2 \circ \gamma_{t,R} :X_R \rightarrow \H^{-1/2}
_{\div_{\partial B_R}}(\partial B_R)$ is compact where $\gamma_{t,R} \uu = \hat x \times\uu|_{\partial B_R}$ for all $\uu 
\in V_{\hh,R}$.

We now have all the tools we need to conclude the proof. Let us build a solution that decomposes as  $\hh = \hh_0 + 
\nabla p$ where $p \in  H_0^1(\Om_R)$. If $\hh$ solves \eqref{eq:varformglo} then  $p$ has to solve $A_R p=0$ (there is 
no source term in $\Omega$) and therefore, $p=0$. As a consequence, $\hh_0$ has to solve
\begin{equation}
\label{eq:operateurequationh0}
(C_R + K_R) \hh_0 = F
\end{equation}
where  the operators $C_R : X_R \rightarrow X_R$ and $K_R:  X_R \rightarrow X_R$  are defined by
\begin{align*}
(C_R\vv,\ww)_{V_{\H,R}} := \int_{\Om_R} \rot\ \vv \cdot \rot\ \b\ww& + \vv\cdot \b\ww\dx  -i\om\<\Z \vv, \ww\>_{\Vg^*,\Vg}\\
&-i \om \<S_1(\hat{x}\times\vv), \ww\>_{\H^{-1/2}_{\div_{\partial B_R}}(\partial B_R),\H^{-1/2}_{\text{curl}_{\partial B_R}}
(\partial B_R)}  \\
(K_R\vv,\ww)_{V_{\H,R}} :=  -(\om^2+1)  \int_{\Om_R}\vv\cdot \b\ww\dx & - i\om\<S_2(\hat{x}\times\vv),\ww\>_{\H^{-1/2}
_{\div_{\partial B_R}}(\partial B_R),\H^{-1/2}_{\text{curl}_{\partial B_R}}(\partial B_R)} 
\end{align*}
for all $\vv,\ww  \in  X_R$ and $F$ is such that $(F,\ww)_{V_{\hh,R}} := i\om \< \ff,\ww\>_{\Vg^*,\Vg} $ for all  $\ww \in X_R
$. From \eqref{eq:CoercZweak} and the properties of $S_1$ and $S_2$ we deduce that $C_R$ is coercive on $X_R$ 
and $K_R$ is compact. The general uniqueness result Theorem \ref{th:MaxUnicite} ensures that it exists $\hh_0$ that 
solves \eqref{eq:operateurequationh0} and that depends continuously on $F$. Therefore, it exists a unique $\hh = \hh_0$ 
that solves the variational formulation \eqref{eq:varformglo} and we obtain the desired result.

\bibliographystyle{plain}
\bibliography{GIBCDirect}{}

\end{document}